\documentclass[fleqn,preprint,3p,a4paper]{elsarticle}

\usepackage{amssymb}
\usepackage{amsmath}
\usepackage{amsthm}
\usepackage{dcolumn}
\usepackage{endnotes}
\usepackage{tabularx}
\usepackage[matrix,arrow]{xy}
\usepackage{wasysym}

\newtheorem{theorem}{Theorem}[section]
\newtheorem{proposition}[theorem]{Proposition}
\newtheorem{lemma}[theorem]{Lemma}
\newtheorem{corollary}[theorem]{Corollary}

\theoremstyle{definition}
\newtheorem{definition}[theorem]{Definition}
\newtheorem{example}[theorem]{Example}
\newtheorem{remark}[theorem]{Remark}
\newtheorem{question}[theorem]{Question}

\newcommand{\II}{\begin{enumerate}}
\newcommand{\III}{\end{enumerate}}

\newcommand{\ir}{{\mathsf{Irr}}}

\newcommand{\ua}{\mathord{\uparrow}}

\newcommand{\mk}{\mathord{Q}}
\newcommand{\wdd}{\mathord{\mathsf{WD}}}

\journal{}

\begin{document}

\begin{frontmatter}



\title{Some new results on well-filteredness of $T_0$-spaces\tnoteref{t1}}
\tnotetext[t1]{This research was supported by the National Natural Science Foundation of China (Nos. 12471070, 12071199)}

\author[X. Xu]{Xiaoquan Xu}
\ead{xiqxu2002@163.com}
\address[X. Xu]{School of Mathematics and Statistics, Minnan Normal University, Zhangzhou 363000, China}

\begin{abstract}
For a $T_0$-space $X$, let $Q (X)$ be the poset of nonempty compact saturated sets of $X$ with the reverse inclusion order. The space $X$ is said to have property Q if it satisfies the following two conditions: (1) $\wedge K$ exists for any $K\in Q(X)$, and (2) for any filtered family $\{K_d : d\in D\}\subseteq Q(X)$ and $x\in X$, if $\bigvee^{\uparrow}_{d\in D}\bigwedge K_d$ exists and $x\not\leq \bigvee^{\uparrow}_{d\in D}\bigwedge K_d$, then there is $\varphi\in \prod\limits_{d\in D}\!\!K_d$ and an upper bound $u$ of $\varphi(D)$ such that $x\not\leq u$. In this paper, we prove that every $d$-space with property Q is well-filtered and the Smyth power space of a $T_0$-space always has property Q. Hence the Smyth power construction preserves the well-filteredness. For a complete lattice $L$ and an order-compatible $d$-topology $\tau$ on it, we show that when $L$ possesses a certain distributivity, $(L, \tau)$ is well-filtered.
\end{abstract}




%% MSC codes here, in the form:

\end{frontmatter}



In domain theory, the $d$-spaces and sober spaces are two important classes of $T_0$-spaces (see \cite{Ershov-1999, GHKLMS-2003, Goubault-2013, Keimel-Lawson-2009, Wyler-1981, Xu-Zhao-2020, Xu-Zhao-2021}). With the development of domain theory and non-Hausdorff topology, the well-filteredness also emerged as a very useful and important property for $T_0$-spaces (see [3, 5, 8, 9, 11, 12, 15, 17-28]). Several characterizations of well-filtered spaces were given in \cite{Xi-Lawson-2017, Xi-Zhao-2017, Xu-Shen-Xi-Zhao-2020-2, Xu-Zhao-2020}.

%Rudin's Lemma plays a crucial role in domain theory (see \cite{Abramsky-Jung-1994, GHKLMS-2003, Gierz-Lawson-Stralka-1983, Goubault-2013}). In \cite{Heckmann-Keimel-2013}, Heckmann and Keimel presented a topological variant of Rudin's Lemma. It has been used to study the various aspects of well-filtered spaces and sober spaces (see \cite{Heckmann-Keimel-2013, Xu-2021-1, Xu-Shen-Xi-Zhao-2020-2, Xu-Wen-Xi-2023, Xu-Zhao-2020, Xu-Zhao-2021}).

%Based on the topological Rudin lemma and Xi and Lawson's work \cite{Xi-Lawson-2017} on well-filtered spaces, Xu and Zhao \cite{Xu-Zhao-2020} obtained a series of characterizations of well-filtered spaces and sober spaces. Such characterizations led them to introduce and study a new class of spaces --- strong $d$-spaces. As a strengthened version of $d$-spaces, the strong $d$-space possesses some important properties (see \cite{Xu-Zhao-2020, Xu-Zhao-2021}). In particular, Xu and Zhao \cite{Xu-Zhao-2021} showed that if the Scott space $\Sigma~\!\!P$ of a poset $P$ is a strong $d$-space, then it is well-filtered. Lawson and Xu proved that if the Scott space $\Sigma P$ of a poset $P$ is a strong $d$-space and $\Sigma (P\times P)=\Sigma P\times\Sigma P$, then $\Sigma P$ is sober, which generalizes a well-known result which states that the Scott space $\Sigma~\!\!L$ of a complete lattice $L$ is sober if $\Sigma (L\times L)=\Sigma L\times\Sigma L$ (cf. \cite[Corollary II-1.12,  Corollary II-4.15 and its proof]{GHKLMS-2003}).

The main purpose of this paper is to give some sufficient conditions for a $T_0$-space to be well-filtered by introducing and studying a new property of $T_0$-spaces --- property Q. The paper is organized as follows.

Section 1 provides some fundamental definitions and notations about topology and lattice-ordered structures which will be used in the whole paper. Also a few basic properties of the poset of
all nonempty compact saturated subsets of a $T_0$-space and some characterizations of $d$-spaces and sober spaces are listed.

In Section 2, we briefly recall the concept and some basic results of well-filtered spaces. A filtered version of topological Rudin lemma is given, which is an important tool for the study of well-filtered spaces.

In Section 3, we introduce the property Q for $T_0$-spaces. It is proved that every $d$-space with property Q is well-filtered and the Smyth power space of a $T_0$-space always has property Q. Therefore, the Smyth power space of a well-filtered space is well-filtered.

Section 4 gives an application of the main result of Section 3. It is shown that for a complete lattice $L$ and an order-compatible $d$-topology $\tau$ on it, if $L$ possesses a certain distributivity, $(L, \tau)$ is well-filtered. A question about the well-filteredness of order-compatible $d$-topologies on a complete lattice is posed.

\section{Preliminaries}

For a poset $P$ and $A\subseteq P$, let
$\mathord{\downarrow}A=\{x\in P: x\leq  a \mbox{ for some }
a\in A\}$ and $\mathord{\uparrow}A=\{x\in P: x\geq  a \mbox{
	for some } a\in A\}$. For  $x\in P$, we write
$\mathord{\downarrow}x$ for $\mathord{\downarrow}\{x\}$ and
$\mathord{\uparrow}x$ for $\mathord{\uparrow}\{x\}$.  A subset $A$
is called a \emph{lower set} (resp., an \emph{upper set}) if
$A=\mathord{\downarrow}A$ (resp., $A=\mathord{\uparrow}A$). Let $P^{(<\omega)}=\{F\subseteq P : F \mbox{~is a nonempty finite set}\}$ and $\mathbf{Fin}~\!P=\{{\uparrow} F : F\in P^{(<\omega)}\}$. For a nonempty subset $B$ of $P$, let $max (B)=\{b\in B : b \mbox{~ is a maximal element of~} B\}$ and $min (B)=\{b\in B : b \mbox{~ is a minimal element of~} B\}$.  An element $u$ is said to be an \emph{upper bound} of a set $C$, if $c\leq u$ for all $c\in C$ (i.e., $C\subseteq {\downarrow} u$). The set of all upper bound of $C$ is denoted by $C^{\uparrow}$. For $E\subseteq P$, if the set of upper bounds of $E$ in $P$ has a unique smallest element, we call this element the \emph{least upper bound} and write it as $\vee E$ or sup~\!\! $E$ (for \emph{supremum}). Dually, the \emph{greatest lower bound} of $E$ in $P$ is written as $\wedge E$ or inf~\!\! $E$ (for \emph{infimum}). The poset $P$ is called an \emph{inf semilattice} (shortly \emph{semilattice}) if for any two elements
$a, b\in P$, $a\wedge b$ exists in $P$. Dually, $P$ is a \emph{sup semilattice} if for any two elements $a, b\in P$, $a\vee b$ exists in $P$.
The set of all natural
numbers is denoted by $\mathbb{N}$. Let $\mathbb{N}^+=\mathbb{N}\setminus \{0\}$. For a set $X$, let $|X|$ be the cardinality of $X$ and $2^X$ the set of all subsets of $X$.

A nonempty subset $D$ of a poset $Q$ is \emph{directed} if every two
elements in $D$ have an upper bound in $D$. The set of all directed sets of $Q$ is denoted by $\mathcal D(Q)$.  The poset $Q$ is called a \emph{directed complete poset}, or \emph{dcpo} for short, if for any
$D\in \mathcal D(Q)$, its supremum $\vee D$ exists in $Q$. The notion $x=\bigvee^{\uparrow}G$ is a convenient device to express that, firstly, the set $G$ is directed and, secondly, $x$ is its leat upper bounded.

%The poset $Q$ is called a \emph{bounded complete dcpo}, if $Q$ is a dcpo and every subset of $Q$ that is bounded above has a least upper bound.

%\begin{lemma}\label{lem-bounde-complete-complete-semilattice} (\cite[Proposition O-2.2(iv)]{GHKLMS-2003}) For a poset $P$, the following two conditions are equivalent:
%\begin{enumerate}[\rm (1)]
%\item $P$ is a bounded complete dcpo.
%\item $P$ is a dcpo and every nonempty subset of $P$ has an inf.
%\end{enumerate}
%\end{lemma}

% The set of all natural numbers is denoted by $\mathbb{N}$. Let $\mathbb{N}^{+}=\mathbb{N}\setminus \{0\}$. When $\mathbb{N}$ is regarded as a poset (in fact, a chain), the order on $\mathbb{N}$ is the usual order of natural numbers.

A subset $U$ of a poset $P$ is \emph{Scott open} if
(i) $U=\mathord{\uparrow}U$, and (ii) for any directed subset $D$ for
which $\vee D$ exists, $\vee D\in U$ implies $D\cap
U\neq\emptyset$. All Scott open subsets of $P$ form a topology,
and we call this topology  the \emph{Scott topology} on $P$ and
denote it by $\sigma(P)$. The space $\Sigma~\!\! P=(P,\sigma(P))$ is called the
\emph{Scott space} of $P$. The \emph{upper topology} on a poset $P$, generated
by $\{P\setminus {\downarrow}x : x\in P\}$ (as a subsase), is denoted by $\upsilon (P)$. The upper sets
form the (\emph{upper}) \emph{Alexandroff topology} $\alpha (P)$.

For a $T_0$-space $X$, we use $\leq_X$ to represent the \emph{specialization order} of $X$, that is, $x\leq_X y$ if{}f $x\in \overline{\{y\}}$.  We will use $\Omega~\!\!X$ or even $X$ to denote the poset $(X, \leq_X)$. In what follows, when a $T_0$-space is considered as a poset, the order always refers to the specialization order if no other explanation is given. For a poset $P$, a $T_0$-topology $\tau$ on $P$ is said to be \emph{order}-\emph{compatible} if $\leq_{\tau}$ agrees with the original order on $P$. It is easy to verify that $\tau$ is order-compatible iff $\upsilon (P)\subseteq \tau\subseteq \alpha (P)$. Let $\mathcal O(X)$ (resp., $\Gamma(X)$) be the set of all open subsets (resp., closed subsets) of $X$, and let $\mathcal S^u(X)=\{\ua x : x\in X\}$, $\mathcal S_c(X)=\{\overline{{\{x\}}} : x\in X\}$ and $\mathcal D_c(X)=\{\overline{D} : D\in \mathcal D(X)\}$.

A $T_0$-space $X$ is called a \emph{$d$-space} (or \emph{monotone convergence space}) if $X$ (with the specialization order) is a dcpo and $\mathcal O(X) \subseteq \sigma(X)$ (cf. \cite{GHKLMS-2003, Wyler-1981}). For a set $Y$ and a topology $\tau$ on $Y$, we call $\tau$ a $d$-\emph{topology} on $Y$ if $(Y, \tau)$ is a $d$-space. A topology $\delta$ on a poset $P$ is said to be an \emph{order}-\emph{compatible} $d$-\emph{topology} if $\delta$ is order-compatible and $(P, \delta)$ is a $d$-space. Clearly, $\delta$ is an  order-compatible $d$-topology on $P$ iff $P$ is a dcpo and $\upsilon (P)\subseteq \delta \subseteq \sigma (P)$.

\begin{proposition}\label{prop-d-space-charac} (\cite[Proposition 3.3]{Xu-Shen-Xi-Zhao-2020-2}) For a $T_0$-space $X$, the following conditions are equivalent:
\begin{enumerate}[\rm (1)]
\item $X$ is a $d$-space.
            \item $\mathcal D_c(X)=\mathcal S_c(X)$.
            \item  For any $D\in \mathcal D(X)$ and $U\in \mathcal O(X)$, $\bigcap_{d\in D}\ua d\subseteq U$ implies $\ua d \subseteq U$ (i.e., $d\in U$) for some $d\in D$.
\end{enumerate}
\end{proposition}

%\begin{lemma}\label{lem-d-space-max}
%If $X$ is a $d$-space and $A$ is a nonempty closed subset of $X$, then $\mbox{max}(A)\neq\emptyset$ and $A=\downarrow \mbox{max}(A)$.
%\end{lemma}
%\begin{proof} For $a\in A$, by Zorn's Lemma there is a maximal chain $C$ containing $a$ in $A$. Since $X$ is a $d$-space, $c=\vee C$ exists and $c\in A$. By the maximality of $C$, we have $c\in max (A)$ and $a\leq c$. So $\mbox{max}(A)\neq\emptyset$ and $A=\downarrow \mbox{max}(A)$.
%\end{proof}

A nonempty subset $A$ of a $T_0$-space $X$ is said to be \emph{irreducible} if for any $\{F_1, F_2\}\subseteq \Gamma(X)$, $A \subseteq F_1\cup F_2$ implies $A \subseteq F_1$ or $A \subseteq  F_2$.  Denote by $\ir(X)$ (resp., $\ir_c(X)$) the set of all irreducible (resp., irreducible closed) subsets of $X$. Clearly, every directed subset of $X$ (with the specialization order) is irreducible. The space $X$ is called \emph{sober}, if for any  $A\in\ir_c(X)$, there is a unique point $x\in X$ such that $A=\overline{\{x\}}$.

\begin{proposition}\label{prop-sober-charac} (\cite[Proposition 5.7]{Xu-Shen-Xi-Zhao-2020-2})) For a $T_0$-space $X$, the following conditions are equivalent:
\begin{enumerate}[\rm (1)]
\item $X$ is sober.
            \item $\ir_c (X)=\mathcal S_c(X)$.
            \item  For any $A\in \ir(X)$ and $U\in \mathcal O(X)$, $\bigcap_{a\in A}\ua a\subseteq U$ implies $\ua a \subseteq U$ for some $a\in A$.
\end{enumerate}
\end{proposition}

For the sobrity of upper topology on a poset, we have the following result.

\begin{proposition}\label{prop-upper-topology-WF} (\cite[Proposition 2.9]{Xu-Shen-Xi-Zhao-2020-2}) For a poset $P$, the space $(P, \upsilon (P))$ is sober iff  $\bigvee A$ exists in $P$ for any $A\in\ir ((P, \upsilon (P))$. Therefore, for any complete lattice $L$, $(L, \upsilon (L))$ is sober.
\end{proposition}

%The following two lemmas on irreducible sets are well-known.

%\begin{lemma}\label{irrsubspace}
%Let $X$ be a space and $Y$ a subspace of $X$. Then the following conditions are equivalent for a
%subset $A\subseteq Y$:
%\begin{enumerate}[\rm (1)]
%	\item $A$ is an irreducible subset of $Y$.
%	\item $A$ is an irreducible subset of $X$.
%	\item ${\rm cl}_X A$ is an irreducible subset of $X$.
%\end{enumerate}
%\end{lemma}

%\begin{lemma}\label{irrimage}
%	If $f : X \longrightarrow Y$ is continuous and $A\in\ir (X)$, then $f(A)\in \ir (Y)$.
%\end{lemma}

A subset $A$ of a $T_0$-space $X$ is called \emph{saturated} if it equals the intersection of all open sets containing it (equivalently, $A$ is an upper set with respect to the specialization order). We use $\mathord{Q}(X)$ to
denote the set of all nonempty compact saturated subsets of $X$ and endow it with the \emph{Smyth order} $\sqsubseteq$: $K_1\sqsubseteq K_2$ if{}f $K_2\subseteq K_1$ for all $K_1,K_2\in \mathord{Q}(X)$.

\begin{lemma}\label{lem-sups-in-Smyth} Let $X$ be a $T_0$-space and $\{K_i : i\in I\}\subseteq \mk (X)$ a nonempty family. Then $\bigvee_{i\in I} K_i$ exists in $\mk (X)$ if{}f~$\bigcap_{i\in I} K_i\in \mk (X)$. In this case $\bigvee_{i\in I} K_i=\bigcap_{i\in I} K_i$.
\end{lemma}
\begin{proof} Assume that $\bigvee_{i\in I} K_i$ exists in $\mk (X)$. Let $K=\bigvee_{i\in I} K_i$. Then $K\subseteq K_i$ for all $i\in I$, and hence $K\subseteq \bigcap_{i\in I} K_i$. For any $x\in \bigcap_{i\in I} K_i$, $\ua x$ is a upper bound of $\{K_i : i\in I\}\subseteq \mk (X)$, whence $K\sqsubseteq \ua x$ or, equivalently, $\ua x \subseteq K$. Therefore, $\bigcap_{i\in I} K_i\subseteq K$. Thus $\bigcap_{i\in I} K_i=K\in \mk (X)$.

Conversely, if $\bigcap_{i\in I} K_i\in \mk (X)$, then $\bigcap_{i\in I} K_i$ is an upper bound of $\{K_i : i\in I\}$ in $\mk (X)$. Let $G\in \mk (X)$ be another upper bound of $\{K_i : i\in I\}$, then $G\subseteq K_i$ for all $i\in I$, and consequently, $G\subseteq \bigcap_{i\in I} K_i$, that is, $\bigcap_{i\in I} K_i\sqsubseteq G$. So $\bigvee_{i\in I} K_i=\bigcap_{i\in I} K_i$.
\end{proof}

The following result is well-known and can be easily verified (cf. \cite[Fact 2.6 and Section 3.2]{Heckmann-Keimel-2013}).

\begin{lemma}\label{lem-compact-saturated-in-Alexanderoff-topologyin} Let $P$ be a poset. Then $\ir ((P, \alpha (P))=\mathcal D(P)$ and $Q((P, \alpha (P)))=\mathbf{Fin}~\!P$.
\end{lemma}

For a $T_0$-space $X$ and $U\in \mathcal O(X)$, define $\Box U=\{K\in Q(X) : K\subseteq  U\}$. The \emph{upper Vietoris topology} on $Q(X)$ is the topology that has $\{\Box U : U\in \mathcal O(X)\}$ as a base and the resulting space is denoted by $P_S(X)$, called the \emph{Smyth power space} or \emph{upper space} of $X$ (cf. \cite{Heckmann-1992, Schalk-1993}). It is easy to see that the specialization order on $P_S(X)$ is the Smyth order, that is, for $K_1,K_2\in \mathord{Q}(X)$, $K_1\leq_{P_S(X)}K_2$ if{}f $K_2\subseteq K_1$. The \emph{canonical mapping} $\xi_X: X\longrightarrow P_S(X)$, $x\mapsto\ua x$, is a topological embedding (cf. \cite{Heckmann-1992, Heckmann-Keimel-2013, Schalk-1993}).

The Smyth power space is a very important structure in domain theory, which plays a fundamental role in modeling the semantics of non-deterministic programming languages (see \cite{Abramsky-Jung-1994, GHKLMS-2003, Schalk-1993}).

The following result can be verified straightforwardly (see e.g. \cite[p. 128]{Schalk-1993}
or \cite[proof of Lemma 3.1]{Jia-Jung-2016}).

\begin{lemma}\label{lem-union-operator-continuous} Let $X$ be a $T_0$-space $X$ and $\mathcal K\in\mk(P_S(X))$. Then $\bigcup \mathcal K\in\mk(X)$.
\end{lemma}

For the sobriety of the Smyth power spaces, we have the following well-known result.

\begin{theorem}\label{Schalk-Heckman-Keimel theorem} (Heckmann-Keimel-Schalk Theorem) (\cite[Theorem 3.13]{Heckmann-Keimel-2013} and \cite[Lemma 7.20]{Schalk-1993}) For a $T_0$-space $X$, the following conditions are equivalent:
\begin{enumerate}[\rm (1)]
\item $X$ is sober.
 \item  For any $\mathcal A\in \ir(P_S(X))$ and $U\in \mathcal O(X)$, $\bigcap\mathcal A\subseteq U$ implies $K \subseteq U$ for some $K\in \mathcal A$.
 \item $P_S(X)$ is sober.
\end{enumerate}
\end{theorem}

\section{Topological Rudin lemma and well-filtered spaces}

\begin{definition}\label{def-WF-space} A $T_0$-space $X$ is said to be \emph{well-filtered} if for any open set $U$ and filtered family $\mathcal{K}\subseteq \mathord{Q}(X)$, $\bigcap\mathcal{K}{\subseteq} U$ implies $K{\subseteq} U$ for some $K{\in}\mathcal{K}$.
\end{definition}

The well-filtered space was originally introduced and investigated by Heckmann \cite{Heckmann-1992}, which was called the $\mathcal U_K$-\emph{admitting space} in \cite{Heckmann-1992}.

From the definition of well-filtered spaces, one can directly obtain the following (see \cite[Proposition 3]{Xi-Zhao-2017}).

\begin{proposition}\label{prop-WF-filtered-Smyth-power-d-space} A $T_0$ space $X$ is well-filtered iff its Smyth power space $P_S(X)$ is a $d$-space.
\end{proposition}

By Lemma \ref{lem-sups-in-Smyth} and Proposition \ref{prop-WF-filtered-Smyth-power-d-space}, we get the following corollary.

\begin{corollary}\label{cor-WF-filtered-compact-sets-cap-compact} For a well-filtered space $X$ and a filtered family $\mathcal K\subseteq \mk (X)$, $\bigcap \mathcal K\in \mk (X)$ and $\bigvee^{\uparrow}_{Q(X)}\mathcal K=\bigcap \mathcal K$.
\end{corollary}
%\begin{proof} Clearly, $\bigcap \mathcal K$ is saturated and $\bigcap \mathcal K\neq\emptyset$ (otherwise, $\bigcap \mathcal K=\emptyset$ implies $K=\emptyset$ for some $K\in \mathcal K$, a contradiction). Now we verify that $\bigcap \mathcal K$ is compact. Let $\{U_i : i\in I\}$ be an open cover of $\bigcap \mathcal K$. As $X$ is well-filtered, there is a $K\in\mathcal K$ such that $K\subseteq \bigcup_{i\in I}U_i$. By the compactness of $K$, there is $J\in I^{(<\omega)}$ such that $K\subseteq \bigcup_{i\in J}U_i$, and hence $\bigcap \mathcal K\subseteq K\subseteq \bigcup_{i\in J}U_i$. Thus $\bigcap \mathcal K\in \mk (X)$.
%\end{proof}

For the well-filteredness of Scott space, Xi and Lawson \cite{Xi-Lawson-2017} gave the following useful result.

\begin{proposition}\label{prop-Scott-topology-on-complete-lattice-WF} (\cite[Corollary 3.2]{Xi-Lawson-2017}) For a complete lattice $L$, its Scott space $(L, \sigma (L))$ is well-filtered.
\end{proposition}

Rudin's Lemma, due to Mary Rudin \cite{Rudin-1981}, is an important tool in domain theory and non-Hausdorff topology (see \cite{Abramsky-Jung-1994, GHKLMS-2003, Gierz-Lawson-Stralka-1983, Goubault-2013}). In \cite{Heckmann-Keimel-2013}, Heckmann and Keimel presented a topological variant of Rudin's Lemma.

In this paper we only need the following filtered version of topological Rudin lemma.

\begin{lemma}\label{t Rudin} Let $X$ be a $T_0$-space and $\mathcal K$ a filtered family of nonempty compact saturated sets of $X$. Then every closed set $C$ of $X$  that
meets all members of $\mathcal{K}$ contains an minimal irreducible closed subset $A$ that meets all
members of $\mathcal{K}$.
\end{lemma}

For a $T_0$-space $X$ and $\mathcal{K}\subseteq \mathord{Q}(X)$, let $M(\mathcal{K})=\{A\in \Gamma(X) : K\cap A\neq\emptyset \mbox{~for all~} K\in \mathcal{K}\}$ and $m(\mathcal{K})=\{A\in \Gamma(X) : A \mbox{~is a minimal menber of~} M(\mathcal{K})\}$.

\begin{definition}\label{rudinset} (\cite[Definition 2.1]{Shen-Xi-Xu-Zhao-2019} and \cite[Definition 4.6 and Definition 6.1]{Xu-Shen-Xi-Zhao-2020-2})
		Let $X$ be a $T_0$-space and $A$ a nonempty closed set of $X$.
\begin{enumerate}[\rm (1)]
\item $A$ is said to be  a \emph{Rudin set}, if there exists a filtered family $\mathcal K\subseteq \mathord{Q}(X)$ such that $A\in m(\mathcal K)$ (that is, $A$ is a minimal closed set that intersects all members of $\mathcal K$). Let $\mathsf{RD}(X)=\{A\in \Gamma(X) : A\mbox{~is a Rudin set}\}$.
\item $A$ is said to be  a \emph{well-filtered determined set}, $\wdd$ \emph{set} for short, if for any continuous mapping $ f:X\longrightarrow Y$
into a well-filtered space $Y$, there exists a unique $y_A\in Y$ such that $\overline{f(A)}=\overline{\{y_A\}}$.
Denote by $\mathsf{WD}(X)$ the set of all closed well-filtered determined subsets of $X$.
\end{enumerate}
\end{definition}

\begin{lemma}\label{DRWIsetrelation} (\cite[Proposition 6.2]{Xu-Shen-Xi-Zhao-2020-2})
	Let $X$ be a $T_0$-space. Then $S_c(X)\subseteq\mathcal{D}_c(X)\subseteq \mathsf{RD}(X)\subseteq\mathsf{WD}(X)\subseteq\ir_c(X)$.
\end{lemma}

\begin{proposition}\label{wfrudinc} (\cite[Corollary 7.11]{Xu-Shen-Xi-Zhao-2020-2} and \cite[Proposition 2.14]{Xu-2026})
Let $X$ be a $T_0$-space. Then the following conditions are equivalent:
	\begin{enumerate}[\rm (1)]
		\item $X$ is well-filtered.
		 \item $\mathsf{RD}(X)=\mathcal S_c(X)$.
\item $\wdd (X)=\mathcal S_c(X)$.
	\item  For any $A\in \mathsf{RD}(X)$ and $U\in \mathcal O(X)$, $\bigcap_{a\in A}\ua a\subseteq U$ implies $\ua a \subseteq U$ for some $a\in A$.
\item  For any $A\in \wdd (X)$ and $U\in \mathcal O(X)$, $\bigcap_{a\in A}\ua a\subseteq U$ implies $\ua a \subseteq U$ for some $a\in A$.
                \end{enumerate}
\end{proposition}

By Proposition \ref{prop-d-space-charac}, Proposition \ref{prop-sober-charac}, Lemma \ref{DRWIsetrelation} and Proposition \ref{wfrudinc}, we have the following:

\begin{corollary}\label{cor-sober-WF-d-space} Every sober space is well-filtered and every well-filtered space is a $d$-space.
\end{corollary}

A poset $P$ is said to be \emph{Noetherian} if it satisfies the \emph{ascending chain condition} ($\mathrm{ACC}$ for short): every ascending chain has a greatest member.

\begin{proposition}\label{prop-Noethrian-Alexandroff-topology} (\cite[Proposition 5.4 and Theorem 5.7]{Zhao-Ho-2015} and \cite[Proposition 3.8]{Xu-Shen-Xi-Zhao-2020-1}) For a poset $P$, the following conditions are equivalent:
	\begin{enumerate}[\rm (1)]
		\item $P$ is Noetherian.
\item Every directed subset of $P$ has a largest member.
        \item $P$ is a dcpo and ${\uparrow}x\in \sigma(P)$ for all $x\in P$.
        \item $P$ is a dcpo and $\sigma(P)=\alpha(P)$.
        \item $(P,\alpha(P))$ is a $d$-space.
        \item $(P,\alpha(P))$ is well-filtered.
\item $(P,\alpha(P))$ is sober.
	\end{enumerate}
\end{proposition}

\section{Property Q and well-filteredness}

\begin{definition}\label{def-K-semilattice} A $T_0$-space $X$ is called a \emph{Q}-\emph{semilattice}, if $\wedge G$ exists in $\Omega X$ for any $G\in Q(X)$.
\end{definition}

\begin{remark}\label{rem-property-W} Let $X$ be a $T_0$-space and $A$ a nonempty set of $X$. Then
 \begin{enumerate}[\rm (1)]
 \item $\wedge A$ exists in $X$ iff $\bigwedge {\uparrow} A$ in $X$, and in this case $\wedge A=\bigwedge {\uparrow} A$.

 \item $X$ (with the specialization order) is a $Q$-semilattice iff $\wedge H$ exists in $X$ for any nonempty compact set $H$ of $X$.
 \item If $X$ is a $Q$-semilattice, then $X$ is a semilattice.
     \end{enumerate}
     \end{remark}

     \begin{definition}\label{def-property-Q} A $T_0$-space $X$ is said to have \emph{property Q} if it satisfies the following two conditions:
 \begin{enumerate}[\rm (1)]
 \item  $X$ is a $Q$-semilattice;

 \item for any filtered family $\{K_d : d\in D\}\subseteq Q(X)$ and $x\in X$, if $\bigvee^{\uparrow}_{d\in D}\bigwedge K_d$ exists and $x\not\leq \bigvee^{\uparrow}_{d\in D}\bigwedge K_d$, then there is $\varphi\in \prod\limits_{d\in D}\!\!K_d$ and $u\in (\varphi(D))^{\uparrow}$ such that $x\not\leq u$.
     \end{enumerate}
     \end{definition}

Now we give the main result of this paper.

\begin{theorem}\label{theor-property-W-d-space-WF} Every $d$-space with property Q is well-filtered.
\end{theorem}
\begin{proof} Let $X$ be a $d$-space having property Q, $\{K_d : d\in D\}\subseteq Q(X)$ a filtered family and $U\in \mathcal O(X)$ satisfying $\bigcap_{d\in D}K_d\subseteq U$. If $K_d\nsubseteq U$ for each $d\in D$, then by Lemma \ref{t Rudin}, $X\setminus U$ contains a minimal irreducible closed subset $A$ that still meets all $K_d$. Since $X$ has property Q, $k_d=\bigwedge (K_d\cap A)=\bigwedge {\uparrow} (K_d\cap A)$ exists for each $d\in D$.

{\bf Claim 1.} $\{k_d : d\in D\}\subseteq A$ is directed.

For each $d\in D$, as $A={\downarrow} A$ and $K_d\cap A\neq\emptyset$, we have $k_d=\bigwedge (K_d\cap A)\in A$. By the filteredness of $\{K_d : d\in D\}$, $\{k_d : d\in D\}\in \mathcal D(X)$.

{\bf Claim 2.} $\bigvee^{\uparrow}_{d\in D}k_d$ exists and $\bigvee^{\uparrow}_{d\in D} k_d\in A$.

As $X$ is a $d$-space, $X$ is a dcpo and $A$ is a Scott closet subset of $X$, whence $\bigvee^{\uparrow}_{d\in D}k_d$ exists and $\bigvee^{\uparrow}_{d\in D}k_d\in A$.

{\bf Claim 3.} $\bigvee^{\uparrow}_{d\in D}k_d\in\bigcap_{d\in D}{\uparrow} (K_d\cap A)$.

Let $k=\bigvee^{\uparrow}_{d\in D}k_d$. If $k\not\in\bigcap_{d\in D}{\uparrow} (K_d\cap A)$, then $k\not\in {\uparrow} (K_{d_0}\cap A)$ for some $d_0\in D$. Hence for each $g\in K_{d_0}\cap A$, $g\not\leq k=\bigwedge (K_d\cap A)=\bigwedge {\uparrow}(K_d\cap A)$. Since $X$ has property Q, there is $\varphi_g\in \prod\limits_{d\in D}\!\!{\uparrow}(K_d\cap A)$ and $u(g)\in (\varphi_g(D))^{\uparrow}$ such that $g\not\leq u(g)$. For each $d\in D$, by $\varphi_g(d)\leq u(g)$, we have $\varphi_g(d)\in {\uparrow}(K_d\cap A)\cap {\downarrow}u(g)$, and consequently, $K_d\cap A\cap {\downarrow} u(g)\neq \emptyset$. Then by the minimality of $A$, we get $A=A\cap {\downarrow} u(g)$ or, equivalently, $A\subseteq {\downarrow} u(g)$. Therefore, $A\subseteq \bigcap\limits_{g\in K_{d_0}\cap A}{\downarrow} u(g)$, whence $K_{d_0}\cap A\cap \bigcap\limits_{g\in K_{d_0}\cap A}{\downarrow} u(g)=K_{d_0}\cap A\neq\emptyset$. Select $g^*\in K_{d_0}\cap A\cap \bigcap\limits_{g\in K_{d_0}\cap A}{\downarrow} u(g)$. Then $g^*\leq u(g^*)$, a contradiction. Thus $k\in\bigcap_{d\in D}{\uparrow} (K_d\cap A)$.

By Claim 3 we have $\bigvee^{\uparrow}_{d\in D}k_d\in \bigcap_{d\in D}{\uparrow} (K_d\cap A)\subseteq \bigcap_{d\in D}K_d\subseteq U$, which contradicts Claim 2. Therefore, $K_d\subseteq U$ for some $d\in D$, and this completes the proof that $X$ is well-filtered.
\end{proof}

We all know that the Smyth power space of a sober space is sober (Theorem \ref{Schalk-Heckman-Keimel theorem}). However, whether the the Smyth power space of a well-filtered space is still well-filtered was unknown for along time. In fact, Heckmann asked this question in \cite[4.6 Further classes and open problems]{Heckmann-1992} as an open problem. In the following, as an application of Theorem \ref{theor-property-W-d-space-WF}, we will give an affirmative answer to Heckmann's problem. To this end, we first give the following.

\begin{proposition}\label{prop-Smyth-power-space-property-Q} For a $T_0$-space $X$, its Smyth power space $P_S(X)$ has property Q.
\end{proposition}
\begin{proof} Let $\mathcal K\in\mk(P_S(X))$. Then by Lemma \ref{lem-union-operator-continuous}, $\bigcup \mathcal K\in Q(X)$. It is straightforward to verify that $\bigcup \mathcal K=\bigwedge \mathcal K$ in $Q(X)$ (with the Smyth order). Hence $P_S(X)$ satisfies condition (1) of Definition \ref{def-property-Q}. Now we show that condition (2) of Definition \ref{def-property-Q} holds for $P_S(X)$. Assume that $\{\mathcal K_d : d\in D\}\subseteq Q(P_S(X))$ is filtered, $K\in \mk (X)$, $\bigvee^{\uparrow}_{d\in D}\bigwedge\mathcal K_d$ exists in $\mk (X)$ and $K\not\sqsubseteq \bigvee^{\uparrow}_{d\in D}\bigwedge\mathcal K_d$. Then by Lemma \ref{lem-sups-in-Smyth} and Lemma \ref{lem-union-operator-continuous}, $\bigvee^{\uparrow}_{d\in D}\bigwedge\mathcal K_d=\bigcap\limits_{d\in D}\bigcup \mathcal K_d$ and $\bigcap\limits_{d\in D}\bigcup \mathcal K_d=\bigcup\limits_{\psi\in \prod\limits_{d\in D}\!\!\mathcal K_d}\bigcap\limits_{d\in D}\psi(d)\not\subseteq K$. Hence there is $\varphi\in \prod\limits_{d\in D}\!\!\mathcal K_d$ such that $\bigcap\limits_{d\in D}\varphi(d)\not\subseteq K$. Select $q\in \bigcap\limits_{d\in D}\varphi(d)\setminus K$. Then ${\uparrow} q$ is an upper bound of $\varphi(D)$ in $\mk (X)$ and $K\not\sqsubseteq {\uparrow} q$. Therefore, $P_S(X)$ satisfies condition (2) of Definition \ref{def-property-Q}, and hence $P_S(X)$ has property Q.
\end{proof}

% By  Proposition \ref{prop-WF-filtered-Smyth-power-d-space}, Corollary \ref{cor-sober-WF-d-space}, Theorem \ref{theor-property-W-d-space-WF} and Proposition \ref{prop-Smyth-power-space-property-Q}, we get the following result.

\begin{theorem}\label{Smythwf}
	For a $T_0$ space $X$, the following conditions are equivalent:
\begin{enumerate}[\rm (1)]
		\item $X$ is well-filtered.
        \item $P_S(X)$ is well-filtered.
\end{enumerate}
\end{theorem}
\begin{proof}
(1) $\Rightarrow$ (2): By Proposition \ref{prop-WF-filtered-Smyth-power-d-space} and Proposition \ref{prop-Smyth-power-space-property-Q}, $P_S(X)$ is a $d$-space and has property Q, and hence it is well-filtered by Theorem \ref{theor-property-W-d-space-WF}.

(2) $\Rightarrow$ (1): By Proposition \ref{prop-WF-filtered-Smyth-power-d-space} and Corollary \ref{cor-sober-WF-d-space}.
\end{proof}
%\begin{proof}

%(1) $\Rightarrow$ (2): By Lemma \ref{lem-sups-in-Smyth} and Corollary \ref{cor-WF-filtered-compact-sets-cap-compact}, $\mk (X)$ is a dcpo and $\Box U\in \sigma (\mk (X))$ for any $U\in O(X)$, namely, $P_S (X)$ is a $d$-space). Then by Theorem \ref{theor-property-W-d-space-WF} and Proposition \ref{prop-Smyth-power-space-property-Q}, $P_S(X)$ is well-filtered.

%(2) $\Rightarrow$ (1): Suppose that $\mathcal K\subseteq \mathord{\mathsf K}(X)$ is filtered, $U\in \mathcal O(X)$, and $\bigcap \mathcal K \subseteq U$. Let $\widetilde{\mathcal K}=\{\ua_{\mk (X)}K : K\in \mathcal K\}$. Then $\widetilde{\mathcal K}\subseteq \mk (P_S(X))$ is filtered and $\bigcap \widetilde{\mathcal K} \subseteq \Box U$. By the well-filteredness of $P_S(X)$, there is a $K\in \mathcal K$ such that $\ua_{\mk (X)}K\subseteq \Box U$, and whence $K\subseteq U$, proving that $X$ is well-filtered. (or $P_s(X)$ is well-filtered, then it is a $d$-space, or equivalently, $X$ is well-filtered).
%\end{proof}

The above theorem was first proved by Xu, Xi and Zhao \cite{Xu-Xi-Zhao-2021} using a different method. Our work shows that Theorem \ref{Smythwf} can be viewed as a natural corollary of Theorem \ref{theor-property-W-d-space-WF}.

\section{An application}

In this section, as an application of Theorem \ref{theor-property-W-d-space-WF}, we investigate the well-filteredness of order-compatible $d$-topologies on a complete lattice. It is shown that for a complete lattice $L$ and an order-compatible $d$-topology $\tau$ on it, if $L$ possesses a certain distributivity, $(L, \tau)$ is well-filtered.

\begin{definition}\label{def-DQ-distributive-lattice} Let $L$ be a complete lattice and $\tau$ an order-compatible $T_0$-topology on $P$.
\begin{enumerate}[\rm (1)]
\item $L$ is called \emph{directed}-\emph{arbitrary} \emph{distributive} if it satisfies
\vskip 3mm
(DA)  \qquad \qquad \qquad \qquad \qquad \qquad  $\bigvee^{\uparrow}_{d\in D} \bigwedge A_d=\bigwedge\limits_{\psi\in\prod\limits_{d\in D}\!\!\!A_d}\bigvee\limits_{d\in D}\psi (d)$
\vskip 3mm
\noindent for any filtered family $\{A_d : d\in D\}\subseteq 2^L\setminus \{\emptyset\}$.

\item $L$ is called \emph{directed}-\emph{compact} \emph{distributive} with respect to $\tau$ if it satisfies
\vskip 3mm
($\tau$-DQ)  \qquad \qquad \qquad \qquad \qquad ~~  $\bigvee^{\uparrow}_{d\in D} \bigwedge K_d=\bigwedge\limits_{\psi\in\prod\limits_{d\in D}\!\!\!K_d}\bigvee\limits_{d\in D}\psi (d)$
\vskip 3mm
\noindent for any filtered family $\{K_d : d\in D\}\subseteq Q((L, \tau))$.
\end{enumerate}
\end{definition}

For a complete lattice $L$, the directed-compact distributivity is a relative property. More precisely, it is related to the poset of nonempty compact saturated sets of $(L, \tau)$.  Clearly, (DA) implies ($\tau$-DQ). The converse implication does not hold in general, as shown in the following example.

\begin{example}\label{exam-DQ-not-imply-DA} Let $L=\mathbb{N}^+\cup\{\top, \bot\}$ and define an order on $L$ as follows (see Figure 1):
\begin{enumerate}[\rm (i)]
\item $\bot <n<\top$ for any $n\in\mathbf{N}^+$, and
\item $n$ and $m$ are incomparable for any $n, m\in \mathbf{N}^+$ with $n\neq m$.
\end{enumerate}

% on $P$ by $x\leq_P y$ iff $x=y$ or $y\in \mathbb{N}^+$ and $x=\bot$ or $x\in \mathbb{N}^+$ and $y=\top$ (see Figure 1).

 \begin{figure}[ht]
	\centering
	\includegraphics[height=4cm,width=6cm]{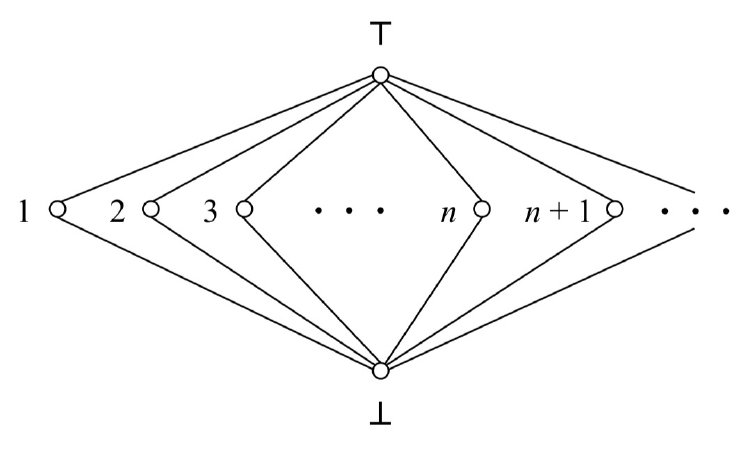}
	\caption{The complete lattice $L$ in Example \ref{exam-DQ-not-imply-DA}}
\end{figure}

 Clearly, $L$ is a complete lattice and it is Noetherian. Consider the Scott space $(L, \sigma (L))$. Then
\begin{enumerate}[\rm (a)]
    \item For any $H\subseteq \mathbb{N}^+$ with $|H|\geq 2$, $\wedge H=\bot$ and $\vee H=\top$.
    \item $\sigma (L)=\alpha (L)$ and $Q((L, \sigma (L))=\mathbf{Fin}~\!L$ by Lemma \ref{lem-compact-saturated-in-Alexanderoff-topologyin}             and Proposition \ref{prop-Noethrian-Alexandroff-topology}.
    \item $L$ is directed-compact distributive with respect to $\sigma (L)$.

    Suppose that $\{{\uparrow}F_d : d\in D\}\subseteq Q((L, \sigma))=\mathbf{Fin}~\!L$ if filtered. We will show that $\bigvee^{\uparrow}_{d\in D}\bigwedge {\uparrow}F_d=\bigwedge\limits_{\psi\in \prod\limits_{d\in D}\!\!{\uparrow}F_d}\bigvee\limits_{d\in D}\psi (d)$.

    \vspace{0.2cm}

    {\bf Case 1:} There is $d_0\in D$ such that ${\uparrow}F_{d_0}={\uparrow} \top=\{\top\}$.

    \vspace{0.1cm}

    As ${\uparrow}F_{d_0}=\{\top\}$, $\bigwedge {\uparrow}F_{d_0}=\top$. So $\bigvee^{\uparrow}_{d\in D}\bigwedge {\uparrow}F_d=\top$. For any $\psi\in \prod\limits_{d\in D}\!\!{\uparrow}F_d$, we have $\psi (d_0)=\top$, whence $\bigvee\limits_{d\in D}\psi (d)=\top$. Therefore, $\bigwedge\limits_{\psi\in \prod\limits_{d\in D}\!\!{\uparrow}F_d}\bigvee\limits_{d\in D}\psi (d)=\top= \bigvee^{\uparrow}_{d\in D}\bigwedge {\uparrow}F_d$.

 \vspace{0.1cm}

    {\bf Case 2:} ${\uparrow}F_{d}={\uparrow} \bot=L$ for all $d\in D$.

 \vspace{0.1cm}

    In this case, $\bigvee^{\uparrow}_{d\in D}\bigwedge {\uparrow}F_d=\bot$. Define $\psi_{\bot}\in \prod\limits_{d\in D}\!\!{\uparrow}F_d$ by $\psi_\bot(d)=\bot$ for each $d\in D$. Then $\bigvee\limits_{d\in D}\psi_\bot (d)=\bot$. Thus $\bigwedge\limits_{\psi\in \prod\limits_{d\in D}\!\!{\uparrow}F_d}\bigvee\limits_{d\in D}\psi (d)=\bot=\bigvee^{\uparrow}_{d\in D}\bigwedge {\uparrow}F_d$.

 \vspace{0.1cm}

    {\bf Case 3:} ${\uparrow}F_{d}\neq \{\top\}$ for all $d\in D$ and there is $d_1\in D$ such that ${\uparrow}F_{d_1}\neq L$.

    \vspace{0.1cm}

    Let $D^\ast=\{d\in D : {\uparrow}F_{d_1}\neq L\}$. Then $d_1\in D^\ast$ and hence $D^\ast\neq\emptyset$. It is straightforward to verify
that $\{{\uparrow}F_d : d\in D^\ast\}$ is filtered. For each $d\in D^\ast$, ${\uparrow}F_d={\uparrow}min(F_d)$ and $min(F_d)\subseteq \mathbb{N}^+$. For $d_1, d_2\in D^\ast$, it is easy to see that ${\uparrow}F_{d_1}\subseteq {\uparrow}F_{d_2}$ iff $min(F_{d_1})\subseteq min(F_{d_2})$. Therefore, $\{min(F_d) : d\in D^\ast\}\subseteq (\mathbb{N}^+)^{(<\omega)}$ is filtered, and hence it has a least element, say $min(F_{d_2})$ ($d_2\in D^\ast$). Then ${\uparrow}F_{d_2}={\uparrow}min(F_{d_2})\subseteq {\uparrow}min(F_d)={\uparrow}F_d$ for any $d\in D$. It follows that  $\bigvee^{\uparrow}_{d\in D}\bigwedge {\uparrow}F_d=\bigwedge {\uparrow}F_{d_2}=\bigwedge min(F_{d_2})$. For each $u\in min(F_{d_2})$, define $\psi_{u}\in \prod\limits_{d\in D}\!\!{\uparrow}F_d$ by $\psi_u(d)=u$ for each $d\in D$. Then $\bigvee\limits_{d\in D}\psi_u (d)=u$. So $\bigwedge min(F_{d_2})=\bigvee^{\uparrow}_{d\in D}\bigwedge {\uparrow}F_d\leq \bigwedge\limits_{\psi\in \prod\limits_{d\in D}\!\!{\uparrow}F_d}\bigvee\limits_{d\in D}\psi (d)\leq \bigwedge\limits_{u\in min(F_{d_2})} \bigvee\limits_{d\in D}\psi_u (d)=\bigwedge min(F_{d_2})$, and hence $\bigwedge\limits_{\psi\in \prod\limits_{d\in D}\!\!{\uparrow}F_d}\bigvee\limits_{d\in D}\psi (d)=\bigvee^{\uparrow}_{d\in D}\bigwedge {\uparrow}F_d$.

So to sum up, we have $\bigvee^{\uparrow}_{d\in D}\bigwedge {\uparrow}F_d=\bigwedge\limits_{\psi\in \prod\limits_{d\in D}\!\!{\uparrow}F_d}\bigvee\limits_{d\in D}\psi (d)$. Thus $L$ is directed-compact distributive with respect to $\sigma (L)$.

    \item $L$ is not directed-arbitrary distributive.

    For each $n\in \mathbb{N}^+$, let $A_n=\mathbb{N}^+\setminus \{1, 2, ..., n\}$. Then $\{A_n : n\in\mathbb{N}^+\}\subseteq 2^L\setminus \{\emptyset\}$ is filtered and $\bigwedge A_m=\bot$ for all $m\in\mathbb{N}^+$. Hence $\bigvee^{\uparrow}_{n\in \mathbb{N}^+}\bigwedge A_n=\bot$. For any $\psi\in \prod\limits_{n\in \mathbb{N}^+}\!\!A_n$, we have $\psi (\psi(1))\in \mathbb{N}^+\setminus \{1, 2, ..., \psi (1)\}$, whence $\psi (\psi (1))\neq \psi (1)$ (indeed, $\psi(\mathbb{N}^+$ is infinite). So $\bigvee\limits_{n\in \mathbb{N}^+}\psi (n)=\top$ by (a). Therefore, $\bigwedge\limits_{\psi\in \prod\limits_{n\in \mathbb{N}^+}\!\!\!A_n} \bigvee\limits_{n\in \mathbb{N}^+}\psi (n)=\top > \bot=\bigvee^{\uparrow}_{n\in \mathbb{N}^+}\bigwedge A_n$. Thus $L$ is not directed-arbitrary distributive.
\end{enumerate}
\end{example}

\begin{theorem}\label{theor-DQ-WF} Let $L$ be a complete lattice and $\tau$ an order-compatible $d$-topology on $L$. If $L$ is directed-compact distributive with respect to $\tau$, then $(L, \tau)$ is well-filtered.
\end{theorem}
\begin{proof} As $L$ is a complete lattice, $L$ is a $Q$-semilattice. Now we show that $(L, \tau)$ satisfies condition (2) of Definition \ref{def-property-Q}. Let $\{K_d : d\in D\}\subseteq Q((L, \tau))$ be filtered and $x\in L$ satisfying $x\not\leq \bigvee^{\uparrow}_{d\in D}\bigwedge K_d$. Then by the directed-compact distributivity of $L$, we have $x\not\leq \bigvee^{\uparrow}_{d\in D}\bigwedge K_d=\bigwedge\limits_{\psi\in\prod\limits_{d\in D}\!\!\!K_d}\bigvee\limits_{d\in D}\psi (d)$, and consequently,  there is $\varphi\in \prod\limits_{d\in D}\!\!K_d$ such that $x\not\leq \bigvee\limits_{d\in D}\psi (d)$. Let $q=\bigvee\limits_{d\in D}\psi (d)$. Then $q\in (\varphi (D))^{\uparrow}$ and $x\not\leq q$. Therefore, condition (2) of Definition \ref{def-property-Q} holds for $(L, \tau)$, and hence $(L, \tau)$ has property Q. By Theorem \ref{theor-property-W-d-space-WF}, $(L, \tau)$ is well-filtered.
\end{proof}

\begin{corollary}\label{cor-DA-WF} Let $L$ be a complete lattice and $\tau$ an order-compatible $d$-topology on $L$. If $L$ is directed-arbitrary distributive, then $(L, \tau)$ is well-filtered.
\end{corollary}

Theorem \ref{theor-DQ-WF} and Corollary \ref{cor-DA-WF} can be restated as follows.

\begin{theorem}\label{theor-DQ-WF-topological-version} Let $X$ be a $d$-space. If $\Omega~\!\! X$ is a complete lattice and directed-compact distributive with respect to the topology of $X$, then $X$ is well-filtered.
\end{theorem}

\begin{corollary}\label{cor-DA-WF-topological-version} Let $X$ be a $d$-space. If $\Omega~\!\! X$ is a complete lattice and directed-arbitrary distributive, then $X$ is well-filtered.
\end{corollary}

By Proposition \ref{prop-upper-topology-WF}, Proposition \ref{prop-Scott-topology-on-complete-lattice-WF} and Corollary \ref{cor-sober-WF-d-space}, the upper topology $\upsilon (L)$ and Scott topology $\sigma (L)$ on a complete lattice $L$ are well-filtered. But we do not know whether any order-compatible $d$-topology on $L$ is well-filtered. So we finally pose the following question.

\begin{question}\label{ques-complete-lattice-d-topology-WF} Is every order-compatible $d$-topology on a complete lattice a well-filtered topology? Or equivalently, is there a complete lattice $L$ and a topology $\upsilon (L) \varsubsetneqq \tau \varsubsetneqq \sigma (L)$ such that $(L, \tau)$ is not well-filtered?
\end{question}

\end{document}